\documentclass[12pt,reqno]{amsart}

\input{ig-header.tex}

\newcommand{\papertitle}{\large{Derangements 
and the $p$-adic incomplete \\[.35em]gamma function}}
\newcommand{\papershorttitle}{DERANGEMENTS AND THE $p$-ADIC INCOMPLETE GAMMA FUNCTION}

\title[\papershorttitle]{\papertitle}
\date{\today}

\author{Andrew O'Desky}
\address{
\parbox{0.5\linewidth}{
    Department of Mathematics\\
    Princeton University\\
    Princeton, NJ 08544-1000, USA\\[.2em]
    \emph{URL}: 
    
    \url{https://www.andrewodesky.com/}\\[-.5em]
    }
}
\email{andy.odesky@gmail.com}

\author{David Harry Richman}
\address{
\parbox{0.5\linewidth}{
    Department of Mathematics\\
    University of Washington\\
    Seattle, WA 98195-4350, USA\\[-.5em]
    }
}
\email{hrichman@uw.edu}

\begin{document}

\begin{abstract}
We introduce a $p$-adic analogue of 
the incomplete gamma function. 
We also introduce quantities ($m$-values) 
    associated to a function on natural numbers 
    and prove a new characterization of $p$-adic continuity 
    for functions with $p$-integral $m$-values. 
Combinatorial interpretations for 
    the integral values of the incomplete gamma function and 
    functions with $m$-values zero or one are obtained, 
which show that these functions count derangements in 
    generalized symmetric groups 
    and permutations with restricted cycle lengths. 
\end{abstract}



\maketitle

\section{Introduction}\label{sec:one} 

The first object of this paper is to introduce 
a $p$-adic analogue of 
the incomplete gamma function. 
The incomplete gamma function $\icg[s,z]$ 
is the holomorphic function of two variables 
obtained by integrating the differential 
$z^{s-1} e^{-z} \,dz$ 
over any contour 
from $z$ to $+\infty$ 
disjoint from the real interval 
$\{-\infty< x \leq 0\}$. 
For positive integers $n$ and $r$, 
$\Gamma(n,r)$ is a rational multiple of $1/e^r$. 
Let $\tau_p \colon \QQ(e) \to \QQ_p$ 
denote the unique field homomorphism satisfying  
$\tau_p(1/e) = 
    1+ p + p^2/2!+p^3/3!+\cdots.$ 
Then $\tau_p\icg[n,r]$ is a $p$-adic number 
and we may ask for the existence of a 
continuous function which interpolates these values. 

\begin{theorem}\label{thm:introincompletegamma} 
There exists a continuous function 
    $$\Gamma_{p}\colon\ZZ_p \times (1+p\ZZ_p)\to \ZZ_p$$ 
    satisfying 
$$\Gamma_{p}(n,r) = 
    \tau_p\icg[n,r]
    $$ 
for any positive integers $n$ and $r$ satisfying $r \equiv 1 \Mod p$. 
Explicitly, for any $s \in \ZZ_p$ and $r \in 1 + p\ZZ_p$ we have the formula 
\begin{equation} 
    \Gamma_p(s,r) = \exp_p(pr) \sum_{k=0}^\infty r^{s-1-k}  k! \binom{s-1}{k} 
\end{equation} 
where $\exp_p$ is the $p$-adic exponential function.  
\end{theorem} 


We also prove a curious formula for $\Gamma_p$ 
involving the floor function $\floor{\cdot}$. 
For any $n\in \NN$ and $r \in \ZZ$ satisfying $r \equiv 1 \Mod p$, 
\begin{equation} 
    \Gamma_p(n+1,1/r) = 
    r^{-n}\exp_p(p/r)  \cdot
    \begin{cases}
        \floor{e^{1/r}r^nn!+\tfrac12} &\text{if $r<0$,} \\
        \floor{e^{1/r}r^nn!} &\text{if $r>0$.}
    \end{cases}
\end{equation} 
This formula expresses 
a $p$-adically continuous function in terms of a function  
which crucially uses the linear ordering on the real numbers. 


\subsection{\texorpdfstring{$p$}{p}-adic continuity via \texorpdfstring{$m$}{m}-values} 

The second object of this paper 
is to present a new type of characterization for 
$p$-adic continuity in terms of 
certain auxiliary quantities associated to 
a function $f \colon \NN \to \QQ_p$. 
These quantities are defined using 
    universal polynomials 
    from the theory of symmetric functions. 
Let 
$h_k = \sum_{i_1 \leq \cdots \leq i_k} x_{i_1} \cdots x_{i_k}$ 
denote the $k$th complete homogeneous symmetric function. 
For each positive integer $k$ 
there is a unique polynomial $M_k(h_1,\ldots,h_k)$ satisfying 
\begin{equation} 
    M_k(h_1,\ldots,h_k) = p_k = 
    x_1^k + x_2^k + x_3^k + \cdots . 
\end{equation} 
The \defn{$k$th $m$-value} 
    of a function $f$ is defined by 
\begin{equation}\label{eqn:mvalues} 
    m_k =
    M_k\left(\tfrac{f(1)}{1!},\tfrac{f(2)}{2!},\ldots,\tfrac{f(k)}{k!}\right).
\end{equation} 
When $f(0) = 1$, the $m$-values of $f$ satisfy the identity 
$$\sum_{n=0}^\infty f(n) \frac{X^n}{n!} = \exp\left(\sum_{k=1}^\infty m_k\frac{X^k}{k}\right).$$
For instance, 
$m_2 = f(2) - f(1)^2$, and 
$m_3 = 
\tfrac{1}{2}f(3) - \tfrac32f(1)f(2) + f(1)^3$ 
(cf. \S\ref{sec:univpolys}). 

\begin{theorem}\label{thm:intromkpinterpolation} 
Let $f \colon \NN \to \QQ_p$ 
    be a function satisfying\footnote
    {For the purpose of testing continuity of $f$ 
we may assume $f(0) = 1$ 
by composing with 
a suitable affine transformation, so 
this hypothesis entails no loss of generality.} 
    $f(0) = 1$. 
Suppose $m_k$ is $p$-integral 
    for all $k\geq 2$. 
Then $f$ is $p$-adically continuous  
    if and only if 
    \begin{equation}\label{eqn:nvaluescongruence}
    m_p \equiv m_1 -1 \Mod{p}.
    \end{equation}
\end{theorem}


One interesting function encompassed by this theorem is  
Morita's $p$-adic gamma function 
    $\Gamma_{p,\mathrm{M}}(n) = 
    (-1)^{n} \prod_{1 \leq k < n,p \nmid k}k$; 
    see Example~\ref{example:morita}. 

Another example comes from the EGF-coefficients 
    $f_\ell$ of 
    the $\ell$-adic Artin--Hasse exponential function $E_\ell(X)$: 
\begin{equation} 
    E_\ell(X) = \exp\bigg(X + \frac{X^\ell}{\ell} + \frac{X^{\ell^2}}{\ell^2} + \cdots \bigg) = 
    \sum_{n=0}^\infty f_\ell(n) \frac{X^n}{n!} . 
\end{equation} 
The $m$-values of $f_\ell$ are 
\begin{equation} 
m_k = 
\begin{cases}
    0 & \text{if $k$ is not a power of $\ell$,}\\
    1 & \text{if $k$ is a power of $\ell$}, 
\end{cases}
\end{equation} 
and 
\begin{equation} 
m_p-m_1+1 = m_p = 
\begin{cases}
    0 & \text{if $p \neq \ell$,}\\
    1 & \text{if $p = \ell$}.  
\end{cases}
\end{equation} 


\begin{corollary}
    \label{cor:artin-hasse}
For primes $p$ and $\ell$, 
the EGF-coefficients $f_\ell$ of 
the $\ell$-adic Artin--Hasse exponential function 
are $p$-adically continuous if and only if $p \neq \ell$. 
\end{corollary}

The same conclusion holds for 
\emph{truncated} 
$\ell$-adic Artin--Hasse exponential functions 
(cf.~\cite{conrad}). 
For another application of Theorem~\ref{thm:intromkpinterpolation} 
    involving the counts of arrangements and derangements 
    in generalized symmetric groups 
    see Remark~\ref{rmk:anrmkpinterpolation}. 

\subsection{Generalized symmetric groups and cycle-restricted permutations} 

The third object of this paper is 
    to prove combinatorial interpretations 
    for each of the following quantities: 
    the integral values of 
    the incomplete gamma function 
    and functions whose $m$-values 
    are either zero or one. 

\subsubsection{Integral values of the incomplete gamma function} 

Just as the values of the gamma function at positive integers 
count permutations, 
we show that the values of 
the incomplete gamma function at integral values 
count cyclic arrangements and derangements  
    in generalized symmetric groups. 


In analogy with 
the classical formula $\Gamma(n+1) = n!$ 
we show that 
    \begin{equation}
    \label{eqn:gamma-arrangement}
    \icg[n+1, 1/r] 
    =
        \gs{n}{r}r^{-n}e^{-1/r} 
    \end{equation}
for $(n,r) \in\NNpos \times \ZZ$ where   
\begin{equation}
    \label{eqn:defngs}
    \gs{n}{r} := 
    \begin{cases}
        (-1)^n (\text{$\#$ of $|r|$-cyclic derangements of degree $n$}) 
            & \text{if $r < 0$,}\\
        \text{$\#$ of $r$-cyclic arrangements of degree $n$} 
            & \text{if $r > 0$.}
    \end{cases}
\end{equation} 
Here $C_r$ (resp. $S_n$) denotes 
the cyclic group of order $r$  
(resp. the symmetric group of degree $n$), and  
an \defn{$r$-cyclic derangement of degree $n$} is 
an element of the generalized symmetric group $C_r \wr S_n$ 
whose action on $C_r\times[n]$ has no fixed point, while  
an \defn{$r$-cyclic arrangement of degree $n$} is 
a choice of subset $A \subset [n]$ and 
element of $C_r \wr S_{|A|}$ 
(for details see \S\ref{sec:wreath}). 
Assaf \cite{assaf} introduced 
$r$-cyclic derangements 
while $r$-cyclic arrangements 
do not seem to have appeared 
in the combinatorics literature yet. 

\subsubsection{Functions with $m$-values zero or one}


Our second combinatorial interpretation 
is for functions 
    whose $m$-values (cf.~\eqref{eqn:mvalues}) 
    are all in $\{0,1\}$. 
We show that any such function admits 
a combinatorial interpretation 
in terms of cycle-restricted permutations. 

\begin{theorem}\label{thm:combinterpretationmkzeroone}
Let $f \colon \NN \to \QQ$ be a function 
    whose $m$-values 
    are all in $\{0,1\}$. 
    Then $f(n)$ is equal to 
    the number $\crp{L}{n}$ 
    of permutations in $S_n$ 
    whose cycle lengths are in 
    $L=\{r : m_r = 1\}$. 
\end{theorem}
This theorem generalizes a classical formula of 
Chowla--Herstein--Scott \cite{chowla-herstein-scott},  
\begin{equation}
\sum_{n = 0}^\infty 
    \# \{\sigma^\ell = 1 : \sigma \in S_n\} 
    \frac{X^n}{n!}
    = 
    \exp\left( \sum_{d | \ell}\frac{X^d}{d} \right), 
\end{equation}
which is recovered when 
$L$ consists of the divisors of a given integer $\ell$. 


Combining this theorem with 
Theorem~\ref{thm:intromkpinterpolation} 
leads to a vast generalization of 
the observation that $n \mapsto (-1)^nd_n$ 
is $p$-adically continuous. 
Define the count of permutations 
    with `permissible' cycle lengths, 
\begin{equation*} 
\crp{L}{n} := \#\{\sigma \in S_n  \;:\; 
        \text{every cycle length of $\sigma$ is in } L \}
\end{equation*} 
where $L \subset \NN^{> 0}$ is any subset.\footnote
{The unique element of $S_0$ acting on 
the empty set is considered 
to have no cycle lengths; 
thus if $L$ is empty, then 
$d_n^L = \delta_{0,n}$ which is never $p$-adically continuous.} 

\begin{theorem}\label{thm:intro-p-derangement}
Let $L$ be a set of positive integers, and let $p \geq 3$ be a prime. 
\begin{enumerate}[(a)]
\item 
$n \mapsto \crp{L}{n}$ is $p$-adically continuous 
    if and only if $1\in L$ and $p \not\in L$.
\item 
$n \mapsto (-1)^n \crp{L}{n}$ is $p$-adically continuous if and only if
$1 \not\in L$ and  $p  \in L$.

\item 
$n \mapsto \crp{L}{n}$ is $2$-adically continuous if and only if 
one of the following hold:
\begin{itemize}
\item 
$1 \in L$ and $2 \not \in L$, or

\item 
$1\not\in L$ and $2 \in L$.
\end{itemize}
\end{enumerate}
\end{theorem}




Theorem~\ref{thm:intro-p-derangement} 
    is deduced from a more uniform statement, 
    which is that $n \mapsto \alpha^n\crp{L}{n}$ 
    is $p$-adically continuous (for any prime $p$) 
    if and only if 
    $$\alpha( \charfn{L}(1) - \charfn{L}(p))
    \equiv 1 \Mod p$$ 
    where $\alpha$ is an integer and $\charfn{L}$ 
    is the characteristic function of $L$. 

We mention some related articles 
    for the interested reader. 
The function $n \mapsto d_n^{\{1,p\}}$ 
    has some intriguing $p$-adic properties 
    (despite not being $p$-adically continuous) 
    cf. \cite{ishihara-ochiai-yugen-yoshida}, \cite{conrad}. 
Arithmetic aspects of the function $n \mapsto \gs{n}{-1}$ 
    counting derangements 
    have been studied in \cite{sun-zagier}, 
    \cite{sun-wu-zhuang}, 
    \cite{miska}, 
    \cite{serafin}. 
For any $r \in \ZZ$ 
    the function $n \mapsto \gs{n}{r}$ 
    is a pseudopolynomial (in the sense of \cite{hall-1971}) 
    which is simultaneously 
    $p$-adic locally analytic 
    for every prime $p$ 
    (Theorem~\ref{thm:gsinterpolates}). 
Generalizations of pseudopolynomials have been 
recently studied in 
\cite{kuperberg2021pseudopolynomials}, 
\cite{eric2021primary}. 
The assertion that 
$n \mapsto \gs{n}{r}$ is 
$p$-adic locally analytic was 
first proven for $r = -1$ 
(counting derangements) 
in \cite{barsky1981analyse} using a formal Laplace transform, 
and for $r = 1$ (counting arrangements) 
more recently in \cite{pacol-zaharescu-2019}. 

%
%

\subsection*{Notation} 

We use $\NN$ to denote the set of nonnegative integers 
and $\NNpos$ for the set of positive integers. 
We let $\ZZ_p$ (resp. $\QQ_p$) denote 
the ring of $p$-adic integers 
(resp. field of $p$-adic numbers). 
We use $|\cdot|_p$ (resp. $v_p$) to denote 
the field norm 
(resp. $p$-adic additive valuation) on $\QQ_p$. 
We let $\CC_p$ denote the metric completion of 
an algebraic closure of $\QQ_p$. 
The $p$-adic valuation $v_p$ and the norm $|\cdot|_p$ 
on $\QQ_p$ both extend uniquely to $\CC_p$  
and are denoted by the same symbols.  



\section{Preliminaries}\label{sec:two} 

In this section we briefly recall 
some standard results from 
$p$-adic analysis for the reader's convenience. 
For more background 
we refer to \cite{schikhof}. 

Given a function $f \colon \NN \to \CC_p$, 
the basic question of $p$-adic interpolation is whether $f$ 
extends to a continuous mapping 
(Lipschitz function, 
locally analytic function, 
analytic function, etc.) 
$\tilde{f} \colon \ZZ_p \to \CC_p$. 
This question is generally approached 
using finite differences. 
The \defn{$n$th finite difference $\findiff[n]$ of $f$ (at zero)} is 
\begin{equation}\label{eqn:findiff} 
    \findiff[n]:= 
    \sum_{k=0}^n
        (-1)^{n-k}\binom{n}{k}
        f(k) . 
\end{equation} 
This admits an inverse relation given by 
$f(n) = \sum_{k=0}^n \binom{n}{k} \findiff[k]$ 
for $n \geq 0$. 

A classical theorem of Mahler \cite{mahler} says that $f$ extends 
to a continuous function on $\ZZ_p$ if and only if 
$|\findiff[n]|_p \to 0$ as $n\to \infty$. 
More generally, faster rates of $p$-adic decay for 
the finite differences 
$\findiff[n]$ correspond to 
stronger $p$-adic properties of $f$. 

\begin{definition}
    A function $f \colon X \to Y$ of metric spaces 
    is a \defn{metric map} if  
        $d_Y(f(x),f(y))
        \leq 
        d_X(x,y)$ 
        for all $x,y \in X$.
    (Equivalently, $f$ is Lipschitz continuous with constant one.) 
If $D$ is an open disk of $\CC_p$, 
    a function $f \colon D \to \CC_p$ is \defn{analytic} if 
    there are elements $a_0 \in D$ and 
    $c_0,c_1,\ldots \in \CC_p$ such that 
    $f(x) = \sum_{n = 0}^\infty c_n(x - a_0)^n$ 
    for all $x \in D$. 
    If $U$ is an open subset of $\CC_p$, 
    a function $f\colon U \to \CC_p$ 
    is \defn{locally analytic} if for each $a \in U$ 
    there is an open disk $D \subset U$ containing $a$ such that $f|_{D}$ 
    is analytic. 
\end{definition}

\begin{theorem}\label{thm:allinterpolations} 
Let $f \colon \NN \to \CC_p$ be a function 
and let $\findiff[n]$ 
    be given by \eqref{eqn:findiff}.
Then $f$ extends to: 
\begin{enumerate}
    \item a continuous map on $\ZZ_p$ 
        if and only if 
        $$|\findiff[n]|_p \to 0
        \,\,\,\text{as $n \to \infty$,}$$
    \item a metric map on $\ZZ_p$ 
        if and only if 
        $$
        \sup_{n \geq 1} |\findiff[n]|_p\, p^{\floor{\frac{\log n}{\log p}} }
        \leq 1, 
        $$
    \item an analytic function on 
        the closed unit disk 
        if and only if 
        $$|\findiff[n]/n!|_p \to 0
        \,\,\,\text{as $n \to \infty$},$$
    \item a locally analytic function on 
        the closed unit disk 
            if and only if 
            $$ \liminf_{n\to\infty} 
                \left(\frac{\valn[p](\findiff[n])}{n}\right) >0.$$
\end{enumerate}
\end{theorem} 

\begin{proof}
The first assertion is due to Mahler \cite{mahler}. 
For the second 
    see \cite[\S5.1]{robert-2000}, or 
    \cite[Proposition~2]{odesky2021functions} 
    for a generalization. 
The third is \cite[Theorem~54.4]{schikhof}. 
The fourth is due to Amice 
    \cite[III, \S10, Corollaire 2, p.~162]{amice-1964}. 
\end{proof}

\section{\texorpdfstring{$p$}{p}-adic continuity via congruences}\label{sec:fivev5} 


In this section, we give 
    a new characterization of 
    $p$-adic continuity 
    for a class of functions satisfying 
    certain integrality conditions. 
We then apply this characterization to a family of 
combinatorial sequences 
which count permutations with restricted cycle types. 

Our characterization uses universal polynomials 
    from the theory of symmetric functions. 
Let 
$h_k = \sum_{i_1 \leq \cdots \leq i_k} x_{i_1} \cdots x_{i_k}$ 
denote the $k$th complete homogeneous symmetric function. 
For each positive integer $k$ 
there is a unique polynomial $M_k(h_1,\ldots,h_k)$ satisfying
    $M_k(h_1,\ldots,h_k) = 
    p_k = x_1^k + x_2^k + x_3^k + \cdots.$ 

\begin{definition}
Let $f \colon \NN \to \QQ_p$ be a function. 
The \defn{$k$th $m$-value of $f$} is 
\begin{equation}
    m_k \coloneqq 
    M_k\left(\tfrac{f(1)}{1!},\tfrac{f(2)}{2!},\ldots,\tfrac{f(k)}{k!}\right).
\end{equation} 
\end{definition}

The main result of this section is 
the following theorem. 

\begin{theorem}[Theorem~\ref{thm:intromkpinterpolation}]\label{thm:mkpinterpolation} 
Let $f \colon \NN \to \QQ_p$ 
    be a function 
    satisfying $f(0) = 1$. 
Suppose $m_k$ is $p$-integral 
    for all $k\geq 2$. 
Then $f$ is $p$-adically continuous  
    if and only if 
    $$
    M_p\left(\tfrac{f(1)}{1!},\tfrac{f(2)}{2!},\ldots,\tfrac{f(p)}{p!}\right)
    \equiv f(1) - 1 \Mod {p\ZZ_p}.$$ 
\end{theorem}

It is rather surprising that continuity 
    for functions with $p$-integral $m$-values 
    can be detected 
    using just $m_p$ and $m_1$, as these $m$-values only depend 
    on the values of $f$ on the integers $1,2,\ldots,p$. 
Thus the assumption of $p$-integrality of $m$-values 
    rigidly couples the values of $f$, 
    in the sense that 
    if $f$ has $p$-integral $m$-values 
    and $m_p \equiv m_1-1 \Mod p$, 
    then if any value of $f$ 
    on an integer greater than $p$ is modified, 
    the $m$-values of $f$ cease to be $p$-integral.



\begin{example}\label{example:morita}
Morita's $p$-adic gamma function is defined by 
    \begin{equation}\label{eqn:exampledefiningreln} 
        \Gamma_{p,\mathrm{M}}(z) = 
        \lim_{n\to z^+}
        (-1)^{n} \prod_{1 \leq k < n, 
    p \nmid k}k
\end{equation} 
where $n$ approaches $z \in \ZZ_p$ 
    through positive integers. 
Set $f(n) = -\Gamma_{p,\mathrm{M}}(n+1)$. 
Then 
\begin{equation} 
    \sum_{n=0}^\infty f(n)\frac{X^n}{n!}
= \frac{1+X^p}{1+X} e^{-X^p/p},
\end{equation} 
    and from this one finds (for $p$ odd) 
\begin{equation} 
m_k = 
\begin{cases}
    (-1)^k & \text{if $k \not \equiv 0 \Mod p$,}\\
    p-2 & \text{if $k = p$,}\\
    (-1)^k(1-p) & \text{if $k \equiv 0 \Mod p$ 
    and $k > p$.}
\end{cases}
\end{equation} 
    Then $m_p = p-2 \equiv (-1) -1 \Mod p$ 
    so we conclude 
    $n \mapsto -\Gamma_{p,\mathrm{M}}(n+1)$,  
    and therefore 
    $n \mapsto \Gamma_{p,\mathrm{M}}(n)$ also, 
    is $p$-adically continuous. 
\end{example}

\subsection{The universal polynomials \texorpdfstring{$M_k$}{M}}\label{sec:univpolys}

It is well-known that $\{h_k\}_k$ and $\{p_k\}_k$ 
    are algebraic bases of 
    the ring of symmetric functions, 
and the universal polynomials $M_1,M_2,M_3,\ldots$ 
    express the change of basis from the power basis $\{p_k\}_k$ 
    to the complete homogeneous basis $\{h_k\}_k$. 

The $\{M_k\}_{k}$ can be computed using 
    the classical identity 
\begin{equation}\label{eqn:classicalidentityph} 
\sum_{k=1}^\infty
    \frac{p_k}{k}
    X^k
     = 
    \log\left(
    \sum_{n=0}^\infty h_n X^n\right).
\end{equation} 
From this one finds 
{\small
\begin{align} 
\sum_{k=1}^\infty
    M_k
    &\frac{X^k}{k}
    = h_{1} X +\\
    & \left( -h_{1}^{2} + 2 h_{2} \right) \frac{X^{ 2 }}{ 2 }+\\
    & \left( h_{1}^{3} - 3 h_{1} h_{2} + 3 h_{3} \right) \frac{X^{ 3 }}{ 3 }+\\
    & \left( -h_{1}^{4} + 4 h_{1}^{2} h_{2} - 2 h_{2}^{2} - 4 h_{1} h_{3} + 4 h_{4} \right) \frac{X^{ 4 }}{ 4 }+\\
    & \left( h_{1}^{5} - 5 h_{1}^{3} h_{2} + 5 h_{1} h_{2}^{2} + 5 h_{1}^{2} h_{3} - 5 h_{2} h_{3} - 5 h_{1} h_{4} + 5 h_{5} \right) \frac{X^{ 5 }}{ 5 }+ \cdots . 
\end{align} 
}
For a function $f$ satisfying $f(0) = 1$, 
set $h_n = f(n)/n!$ for all $n$, 
and write $f_k= f(k)$ for readability. 
Then 
{\small
\begin{align} 
\sum_{k=1}^\infty
    m_k
    &\frac{X^k}{k}
    = f_{1} X+\\
    & \left( -f_{1}^{2} + f_{2} \right) \frac{X^{ 2 }}{ 2 }+\\
    & \left( f_{1}^{3} - \frac{3}{2} f_{1} f_{2} + \frac{1}{2} f_{3} \right) \frac{X^{ 3 }}{ 3 }+\\
    & \left( -f_{1}^{4} + 2 f_{1}^{2} f_{2} - \frac{1}{2} f_{2}^{2} - \frac{2}{3} f_{1} f_{3} + \frac{1}{6} f_{4} \right) \frac{X^{ 4 }}{ 4 }+\\
    & \left( f_{1}^{5} - \frac{5}{2} f_{1}^{3} f_{2} + \frac{5}{4} f_{1} f_{2}^{2} + \frac{5}{6} f_{1}^{2} f_{3} - \frac{5}{12} f_{2} f_{3} - \frac{5}{24} f_{1} f_{4} + \frac{1}{24} f_{5} \right) \frac{X^{ 5 }}{ 5 }+ \cdots . 
\end{align} 
}

One can expand the power series 
    for the logarithm in \eqref{eqn:classicalidentityph} 
    to find 
    $$\frac{M_k}{k} = \sum_{\lambda \vdash k}
    \frac{(-1)^{\ell(\lambda)+1}}{\ell(\lambda)}
    \binom{\ell(\lambda)}{\lambda[1],\ldots,\lambda[k]}
    \prod_{j=1}^kh_j^{\lambda[j]}$$ 
    and 
    $$\frac{m_k}{k} = \sum_{\lambda \vdash k}
    \frac{(-1)^{\ell(\lambda)+1}}{\ell(\lambda)}
    \binom{\ell(\lambda)}{\lambda[1],\ldots,\lambda[k]}
    \prod_{j=1}^k\left(\frac{f(j)}{j!}\right)^{\lambda[j]}$$ 
    where the sum is taken over 
    partitions $\lambda$ of $k$, 
    $\ell(\lambda)$ is 
    the number of parts of $\lambda$, 
    and $\lambda[j]$ is the number of parts 
    of $\lambda$ equal to $j$. 

The $m$-values of a function may also be computed rapidly from 
the recursive relation ($k \geq 2$) 
$$m_{k} = 
\frac{f(k)}{(k-1)!} - \sum_{j=1}^{k-1} 
    \frac{f(k-j)}{(k-j)!}
     m_{j}.$$ 



\subsection{Proof of Theorem~\ref{thm:mkpinterpolation}} 

The key technical result in the proof of 
Theorem~\ref{thm:mkpinterpolation} 
is Theorem~\ref{thm:msequenceLA}, 
which gives necessary and sufficient conditions 
for the convergence of 
certain infinite products of exponentials. 
The convergence takes place in a $p$-adic Banach algebra $\EGFalg{}$ 
of exponential generating functions 
associated to $p$-adically continuous functions. 

\begin{definition} 
Let $\EGFalg{\QQ_p}$ denote the set of all formal power series 
    $\sum_{n =0}^\infty
            c_n \frac{X^n}{n!} \in \QQ_p[\![X]\!]$
    with the property that 
    $c_n =
    \sum_{k=0}^n
        \binom{n}{k}
        (-1)^{n-k}
        f(k)$ 
    for some continuous function 
    $f \colon \ZZ_p \to \QQ_p$. 
Define the norm of 
    $F= \sum_{n =0}^\infty c_n \frac{X^n}{n!}$ to be 
\begin{equation}\label{eqn:defnnorm} 
    \norm[F] \coloneqq \sup_{n \geq 0} |c_n|_p. 
\end{equation} 
We say that $c_n$ is the \defn{$n$th EGF-coefficient} 
    of $\sum_{n =0}^\infty c_n \frac{X^n}{n!}$. 
\end{definition} 

\begin{proposition}\label{prop:closedundermult}
    $(\EGFalg{},\norm[])$ 
    is a Banach algebra. 
\end{proposition}

\begin{proof}
By Mahler's theorem, $\EGFalg{}$ is equal to 
the set of formal power series 
    $\sum_{n} c_n \frac{X^n}{n!}$ 
    satisfying $|c_n|_p \to 0$. 
It is now clear that $\EGFalg{}$ is 
    the completion of the polynomial ring $\QQ_p[X]$ 
    for the norm \eqref{eqn:defnnorm}, and 
    is therefore complete. 

We must verify that $\norm$ is submultiplicative.
Let $F(X) = \sum_{n=0}^\infty a_n \tfrac{X^n}{n!}$ and 
$G(X)  = \sum_{n=0}^\infty b_n \tfrac{X^n}{n!}$.
We have that 
\begin{equation}\label{eqn:productbinomialconvolution} 
    F(X)G(X) = \sum_{n \geq 0}c_n \frac{X^n}{n!} 
    \quad\text{where}\quad
    c_n = \sum_{k=0}^n \binom{n}{k} a_{n-k} b_{k}.
\end{equation} 
By the ultrametric triangle inequality, 
$$
\norm[F \cdot G] = 
\sup_{n \geq 0} \Big| \sum_{k = 0}^n 
        \binom{n}{k} a_{n-k}b_k \Big|_p
\leq 
\sup_{m \geq 0} |a_m|_p
\cdot
\sup_{n \geq 0} |b_n|_p
=
\norm[F] \norm [G].
$$

The product operation on $\CC_p[X]$ is continuous 
    for the norm topology, and therefore extends 
    uniquely to a product 
    on its completion $\EGFalg{}$. 
We have shown 
$(\EGFalg{\QQ_p}, \norm)$ is a Banach algebra. 
\end{proof}


\begin{definition}
Let $m_1,m_2,\ldots \in \QQ_p$. 
The 
\defn{zeta function associated to $m = (m_1,m_2,\ldots)$} 
is the formal power series defined by
\begin{equation}
\label{eq:zeta}
\zeta_m(X) \coloneqq 
\exp\left( \sum_{k=1}^\infty \frac{m_k}{k} X^k\right). 
\end{equation}
\end{definition}

\begin{theorem}\label{thm:msequenceLA}
Let $m_1,m_2,\ldots \in \QQ_p$ and 
suppose $m_r \in \ZZ_p$ 
    for all $r \geq 2$. 
Then $\zeta_m$ is in $\EGFalg{}$
if and only if 
$m_1 \equiv m_p \Mod p$. 
\end{theorem}

To prove the theorem we establish two lemmas. 
For the proofs, we will make repeated use of 
a well-known formula of Legendre for 
the $p$-adic valuation of $n!$, 
$$v_p(n!) 
= \frac{n - s_p(n)}{p-1} $$
where $s_p(n)$ is the sum of the digits of $n$ in base $p$. 



\begin{lemma}\label{lemma:one}
For $m\in \CC_p$ and $r \in \NNpos$ the power series 
    $\exp({mX^r/r})$ is in $\EGFalg{\QQ_p}$ 
    if and only if 
    $v_p(m) > v_p(r) - \frac{r-1}{p-1}$. 
\end{lemma}

\begin{proof}
For $n\geq 0$ 
    the $n$th EGF-coefficient $c_{n}$ of 
    $\exp({mX^r/r}) =\sum_{n\geq 0}c_nX^n/n!$ is 
    given by
    $$
    c_n = \begin{cases}
    m^k \frac{(rk)!}{r^k k!} &\text{if } n = r k \text{ for $k \in \NN$,} \\
    0 &\text{otherwise.}
    \end{cases}
    $$
Then 
\begin{equation} 
\liminf_{n\to\infty}
    \frac{1}{n}v_p(c_n) =
    \liminf_{k \to \infty}
    \frac{1}{rk}v_p\left(m^k \frac{(rk)!}{r^k k!}\right)
    =\frac{1}{r}\left(v_p(m)-v_p(r) + \frac{r-1}{p-1}\right). 
\end{equation} 
If $v_p(m) - v_p(r) + \frac{r-1}{p-1}>0$,
then $v_p(c_n) \to \infty$ as $n\to \infty$ so 
$\exp({mX^r/r})$ is in $\EGFalg{}$ by Theorem~\ref{thm:allinterpolations}.

To finish we show 
    $\exp({mX^r/r})$ is not in $\EGFalg{}$
    if $v_p(m)- v_p(r) + \frac{r-1}{p-1} \leq 0$.
If $v_p(m) - v_p(r) + \frac{r-1}{p-1}<0$, 
    then $v_p(c_n) \not\to \infty$ 
    so $\exp({mX^r/r})$ is not in $\EGFalg{}$. 
If $v_p(m) - v_p(r) + \frac{r-1}{p-1}=0$ 
    then $v_p(c_n)=(p-1)^{-1}(s_p(rn)-s_p(n))$ 
    where $s_p(n)$ is the sum of 
    the base-$p$ digits of $n$. 
This shows $v_p(c_n)< \infty$ is constant 
    for $n\in \{1,p, p^2, \ldots\}$ so 
    $\exp({mX^r/r})$ is not in $\EGFalg{}$. 
\end{proof}

\begin{lemma}\label{lemma:two}
    Let $m_1,m_2,\ldots \in \ZZ_p$. 
    If $m_1 = m_p = 0$ 
    then the zeta function $\zeta_m$ 
    (cf.~\eqref{eq:zeta}) is in $\EGFalg{}$. 
\end{lemma}

\begin{proof} 
We claim that $v_p(r) - \frac{r-1}{p-1} < 0$ 
    if $r \in \NNpos \backslash \{1,p\}$.  
The claim is clear if $v_p(r) = 0$. 
If $v_p(r) = 1$ then $r > p$, so 
    $v_p(r) -  \frac{r-1}{p-1} < 1 - \frac{p-1}{p-1} = 0$. 
If $v_p(r) = v \geq 2$ then $r \geq p^v$, and 
    $v -  \frac{r-1}{p-1}
    \leq v - \frac{p^v-1}{p-1}
    \leq v-1-2^{v-1}$. 
The quantity $v-1-2^{v-1}$ is always negative, 
    which establishes the claim. 

By the claim, 
    $v_p(m_r) \geq 0 > v_p(r) - \frac{r-1}{p-1}$ 
    whenever $r \in \NNpos \backslash \{1,p\}$, 
    and by Lemma~\ref{lemma:one} 
    we conclude $\exp({m_rX^r/r})$ is in $\EGFalg{}$ 
    for any $r\in \NNpos$. 
    Let $G_N = \prod_{r=1}^N \exp({m_r{X^{r}/r}})$ denote 
    the partial product, 
    which is contained in $\EGFalg{}$ 
    by Proposition~\ref{prop:closedundermult}. 
To prove the lemma 
    it suffices to show that 
    $(G_N)_{N\geq 1}$ is a Cauchy sequence 
    by completeness of $\EGFalg{}$. 

For any positive integer $r$ we have that 
\begin{align*}
    \norm[ \exp\left({\frac{m_r X^r}{r}}\right) ]
    &= \sup_{n \geq 0}
        \left|m_r^n\frac{(rn)!}{r^n n!} \right|_p 
    \leq \sup_{n \geq 0}
        \bigg| \prod_{\substack{1 \leq k \leq rn\\k \not \in r\ZZ}} k \bigg|_p 
         = 1. 
\end{align*}
By submultiplicativity 
this implies that $\norm[G_N] \leq 1$.
Moreover, 
\begin{equation}
	\norm[\exp\left({\frac{m_r X^r}{r}}\right) - 1] 
    = \sup_{n \geq 1}
        \left|m_r^n\frac{(rn)!}{r^n n!} \right|_p 
    \leq \sup_{n\geq 1} \bigg| 
    \prod_{\substack{1 \leq k \leq rn\\k \not \in r\ZZ}} 
        k \bigg|_p 
	= \big| (r-1)! \big|_p.  
\end{equation}
Putting these together, 
    $$\norm[G_{N} - G_{N-1}] \leq 
    \norm[G_{N-1}]
    \norm[\exp\left({\frac{m_{N}X^{N}}{N}}\right) - 1]
    \leq |(N-1)!|_p.$$
By the ultrametric triangle inequality,
whenever $N\leq M$ we have 
$$ \norm[G_M - G_N] \leq 
    \max_{N+1\leq i \leq M} 
    \norm[G_{i} - G_{i-1}] 
\leq 
    \max_{N+1\leq i \leq M} 
    |(i-1)!|_p \leq |N!|_p. $$
The upper bound goes to zero as $N\to\infty$.
This verifies that $(G_N)_{N\geq 1}$ is a Cauchy sequence 
and therefore converges to $\zeta_m \in \EGFalg{}$. 
%
\end{proof} 

\begin{proof}[Proof of Theorem~\ref{thm:msequenceLA}] 
We can factor $\zeta_m$ as   
$$
\zeta_m = 
    \exp(m_1X) \exp\left(\frac{m_pX^p}{p}\right) 
    \prod_{r \neq 1,p} \exp\left( \frac{m_rX^r}{r} \right).
$$

Let $m^{\underline{n}} = m(m-1)(m-2) \cdots (m-n+1)$ denote the $n$th falling factorial of $m$.
Note that $|m^{\underline{n}}|_p \to 0$ as $n\to\infty$ for any $m \in \ZZ_p$;
hence
$$(1-X)^m = \sum_{n \geq 0}\binom{m}{n}(-1)^nX^n
= \sum_{n\geq 0} (-1)^n m^{\underline{n}}\, \frac{X^n}{n!} $$
is in $\EGFalg{}$ for any $m\in\ZZ_p$. 
We also have the power series identity 
$\prod_{r\geq 1}\exp\left(- \frac{X^r}{r}\right)  
    = {1-X}.$ 

We see that $\EGFalg{}$ contains $\zeta_m$ 
    if and only if $\EGFalg{}$ contains 
\begin{align}
    \zeta_m \cdot (1-X)^{m_p} 
    &=
    \exp(m_1X) \exp\left(\frac{m_pX^p}{p}\right)
    \prod_{r\neq 1,p} \exp\left( \frac{m_rX^r}{r} \right)
    \prod_{r\geq 1} \exp\left( \frac{-m_pX^r}{r} \right)\\
    &=
    \exp((m_1-m_p)X)
    \prod_{r \neq 1,p} \exp\left( \frac{(m_r-m_p)X^r}{r} \right).
\end{align}
By Lemma~\ref{lemma:two}, 
    $\prod_{r \neq 1,p} 
    \exp\left( \frac{(m_r-m_p)X^r}{r} \right)$ 
    is an invertible element of $\EGFalg{}$, 
    so $\EGFalg{}$ contains $\zeta_m$ if and only if 
    it contains $\exp((m_1-m_p)X)$. 
By Lemma~\ref{lemma:one},
$\exp((m_1-m_p)X)$ is in $\EGFalg{}$ if and only if
$|m_1 -m_p|_p <1$.
\end{proof} 

\begin{proof}[Proof of Theorem~\ref{thm:mkpinterpolation}]
It is a well-known symmetric function identity that  
\begin{equation} 
\sum_{n=0}^\infty
    h_n X^n
     = 
    \exp\left(
    \sum_{k=1}^\infty
    \frac{p_k}{k}
    X^k
     \right)
\end{equation} 
where 
    $h_n = \sum_{1\leq i_1 \leq \cdots \leq i_n} 
    x_{i_1} \cdots x_{i_n}$ 
    and 
    $p_k = \sum_{i\geq 1}x_i^k $. 
    The set $\{h_n\}_n$ is 
    algebraically independent in $\QQ[\{x_i\}_i]$, 
    and by specializing $h_n \mapsto f(n)/n!$ 
    we obtain 
\begin{equation} 
\sum_{n=0}^\infty
    \frac{f(n)}{n!} X^n
     = 
    \exp\left(
    \sum_{k=1}^\infty
    \frac{m_k}{k}
    X^k
     \right)
    = \zeta_m. 
\end{equation} 
Let $c_n =
\sum_{k=0}^n
    \binom{n}{k}
    (-1)^{n-k}
    f(k)$ be 
    the $n$th finite difference of $f$. 
Then by \eqref{eqn:productbinomialconvolution} 
\begin{equation} 
    1 + c_1X + c_2 \frac{X^2}{2!} + \cdots 
    =
    e^{-X} \zeta_m, 
\end{equation} 
so $f$ is $p$-adically continuous 
if and only if 
    $e^{-X} \zeta_m$ is in $\EGFalg{}$. 
Applying Theorem~\ref{thm:msequenceLA} 
    to $m_1-1,m_2,m_3,\ldots$ shows that 
    $e^{-X} \zeta_m$ is in $\EGFalg{}$ 
    if and only if 
\[
    {f(1)-1 = m_1-1 \equiv m_p \Mod p.}
\tag*{\qedhere} 
\]
\end{proof}

\subsection{Cycle-restricted permutations}

A derangement is defined as a permutation with no cycle of length one.
In this section we consider a generalization 
by placing an arbitrary restriction on the cycle lengths of a permutation.
Let $S_n$ denote the symmetric group on $n$ letters. 

Given a set $L$ of positive integers, let
\begin{equation}
\crp{L}{n} = \#\{ \sigma \in S_n : \text{cycle lengths of $\sigma$ are in }L\}.
\end{equation}
We show that the functions $n \mapsto \crp{L}{n}$
for different choices of $L$ 
are precisely the functions 
    whose $m$-values (cf.~\eqref{eqn:mvalues}) 
    are all in $\{0,1\}$. 


\begin{theorem}[Theorem~\ref{thm:combinterpretationmkzeroone}]\label{thm:combinterpretationmkzeroone2}
Let $f \colon \NN \to \QQ$ be a function 
    whose $m$-values (cf.~\eqref{eqn:mvalues}) 
    are all in $\{0,1\}$. 
    Then $f(n)$ is equal to 
    the number $\crp{L}{n}$ 
    of permutations in $S_n$ 
    whose cycle lengths are in 
    $L=\{r : m_r = 1\}$. 
\end{theorem}

Combining the above theorem 
with Theorem~\ref{thm:mkpinterpolation} 
results in the following 
$p$-adic interpolation result 
for counts of cycle-restricted permutations. 
(Note that $n\mapsto (-1)^n$ is $2$-adically continuous,
so part (c) holds equally for $n\mapsto (-1)^n \crp{L}{n}$.)

\begin{theorem}[Theorem~\ref{thm:intro-p-derangement}]\label{thm:p-derangement}
Let $L$ be a set of positive integers,
and let $p\geq 3$ be prime. 
Let $\crp{L}{n}$ denote
    the number of permutations in $S_n$ 
    whose cycle lengths are in $L$.
\begin{enumerate}[(a)]
\item 
$n \mapsto \crp{L}{n}$ is $p$-adically continuous 
    if and only if $1\in L$ and $p \not\in L$.
\item 
$n \mapsto (-1)^n \crp{L}{n}$ is $p$-adically continuous if and only if
$1 \not\in L$ and  $p  \in L$.

\item 
$n \mapsto \crp{L}{n}$ is $2$-adically continuous if and only if 
one of the following hold:
\begin{itemize}
\item 
$1 \in L$ and $2 \not \in L$, or

\item 
$1\not\in L$ and $2 \in L$.
\end{itemize}
\end{enumerate}
\end{theorem}

This theorem encompasses 
a large class of sequences. 
In the following examples, 
we write $P_L$ for the set of primes 
for which the displayed sequence extends to 
a continuous function on $\ZZ_p$. 

\begin{example}[Squares]
\label{ex:crp2}
If $L = \{1,4,9,16, \ldots\}$ consists of the perfect squares,
then 
    \[ ( \crp{L}{n} )= 
    (1,\, 1,\, 1,\, 1,\, 7,\, 31,\, 91,\, 211,\, 1681,\, 52417,\, \ldots)\]
    and $P_{L} = \{\text{all primes}\}$. 
\end{example}

\begin{example}[Non-squares]
If $L = \NNpos \setminus \{1,4,9,16, \ldots\}$, 
then
    \[ ((-1)^n\crp{L}{n}) = 
    (1,\, 0,\, 1,\, -2,\, 3,\, -44,\, 175,\, -1434,\, 12313,\, -59912,\, \ldots)\]
    and $P_{L} = \{\text{all primes}\}$. 
\end{example}

\begin{example}[Powers of $3$]
\label{ex:crp3}
If $L := \{1, 3, 9, 27, \ldots \}$,
then 
$$
( \crp{L}{n} ) = ( 1,\, 1,\, 1,\, 3,\,9,\,21,\,81,\,351,\,1233,\,46089, \ldots)
$$
    and $P_{L} = \{\text{all primes except $3$}\}$. 
\end{example}

\begin{example}[Complement of powers of $3$]
If $L := \NNpos \setminus \{1, 3, 9, 27, \ldots \}$,
then 
$$
( (-1)^n\crp{L}{n} ) = 
( 1,\, 0,\, 1,\, 0,\, 9,\, -24,\, 225,\, -1224,\, 11025,\, -53136,\, \ldots )
$$
and $P_{L} = \{\text{all primes except $3$}\}$. 
\end{example}

\begin{example}[Classical derangements]
\label{ex:crp1}
If $L := \{2,3,4,5,\ldots\}$,
then 
        $$((-1)^n\crp{L}{n}) = 
        (1,\, 0,\, 1,\, -2,\, 9,\, -44,\, 265,\, -1854,\, 14833,\, -133496,\, \ldots) $$ 
    are the classical derangement numbers (up to sign) 
    and $P_{L} = \{\text{all primes}\}$. 
\end{example}
\begin{example}[Prime numbers and complement]
If $L := \{2,3,5,7,\ldots\}$ consists of the prime numbers,
then
    \[ ((-1)^n\crp{L}{n}) = 
    (1,\, 0,\, 1,\, -2,\, 3,\, -44,\, 55,\, -1434,\, 3913,\, -39752,\, \ldots) \]
    and $P_L = \{\text{all primes}\}$. 
If $L' := \NNpos \setminus L$,
then 
    \[ ( \crp{L'}{n} )= 
    (1,\, 1,\, 1,\, 1,\, 7,\, 31,\, 211,\, 1051,\, 10081,\, 107857,\, \ldots) \]
    and $P_{L'} = \{\text{all primes}\}$ also. 
\end{example}
\begin{example}[Artin--Hasse complement]
    If $L_{(p)} = \{p,p^2,p^3,\ldots\}$,
    then $(-1)^n\crp{L_{(p)}}{n}$ is the $n$th EGF-coefficient of $e^{X}E_p(-X)$  
    and $P_{L_{(p)}}=\{p\}$. 
    For instance, if $p = 3$ we have 
    \[
      e^{X}E_3(-X) 
      =1 - \frac{1}{3}X^{3} + \frac{1}{18}X^{6} - \frac{19}{162}X^{9} + 
      \frac{73}{1944}X^{12} - \frac{181}{29160}X^{15} + O(X^{18}) 
    \]
    and 
    \[ ( (-1)^n\crp{L_{(3)}}{n} ) =  
    (1,\,0,\,0,\,-2,\,0,\,0,\,40,\,0,\,0,\,-42560,\,0,\,0,\,17987200,\,0,\,0, \ldots ). 
    \]
\end{example}

Theorem~\ref{thm:combinterpretationmkzeroone2}
and 
Theorem~\ref{thm:p-derangement} 
are proven using  
the following formal power series identity. 

\begin{proposition}\label{prop:EGFcrp} 
    Let $\crp{L}{n}$ denote 
    the number of permutations in $S_n$
    whose cycle lengths are in $L$.
    Then the exponential generating function 
    of $\crp{L}{n}$ factors as 
    \begin{equation}
        \label{eqn:EGFcpn}
        \sum_{n\geq 0} \crp{L}{n} \frac{X^n}{n!} 
        = \prod_{{r \in L}} 
    \exp\left( \frac{X^r}{r} \right).
    \end{equation}
\end{proposition} 

For the proof it is convenient 
to use combinatorial species 
(cf.~\cite[\S1]{nelson1990} 
for a quick introduction 
or \cite{bll} for a comprehensive treatment). 
In the next proof, 
``generating function of a species'' 
always means the exponential generating function. 

\begin{proof} 
For a positive integer $r$, 
consider the subset of permutations in $S_n$ 
    equal to an $r$-cycle (empty if $n \neq r$). 
We have 
\[ \#\{ \sigma \in S_n \;:\; \sigma \text{ is an $r$-cycle}\} = \begin{cases}
(r-1)! &\text{if }n = r ,\\
0 & \text{otherwise}.
\end{cases}\]
The generating function of this sequence is 
$ \tfrac1{r}X^r$.
By the theory of species, $\exp({X^r/r})$ is 
    the generating function of the sequence 
\[ 
n \mapsto \#\{ \sigma \in S_n \;:\; \sigma \text{ decomposes into disjoint $r$-cycles}\}. \]

A standard result from 
    the theory of species says that
the product of 
    the generating functions of two species $A$ and $B$ 
is the generating function of the species where we partition $[n]$ 
into disjoint subsets $S \sqcup T$, and impose an $A$-species on $S$ and a $B$-species on $T$.
This generalizes by induction to finite products. 

Consider the species 
    $$ 
    [n] \mapsto 
    \begin{cases}
        \{ \sigma \in S_n \;:\; \text{every cycle length of $\sigma$ is in } L\} 
        & \text{ if $n \leq N$,}\\
        \varnothing & \text{ otherwise.} 
    \end{cases}
    $$
The generating function of this species is 
    $\sum_{n = 0}^N \crp{L}{n} \frac{X^n}{n!}$. 
By the previous paragraphs, 
    this generating function factors as 
\[ \sum_{n= 0}^N \crp{L}{n} \frac{X^n}{n!}
    = \prod_{\substack{r \in L \\ r\leq N}} 
    \exp\left(\frac{X^r}{r}\right)
    = \prod_{r \in L} 
    \exp\left(\frac{X^r}{r}\right)
    \Mod{X^{N+1}}. 
\]
    Taking the limit as $N \to \infty$ 
    in the $X$-adic topology 
    verifies \eqref{eqn:EGFcpn}.
\end{proof} 

\begin{proof}[Proof of Theorem~\ref{thm:p-derangement}]
We apply the criterion of Theorem~\ref{thm:mkpinterpolation},
that a function on $\NN$ with integral $m_k$-values for $k \geq 2$ 
    is $p$-adically continuous 
    if and only if its $m$-values satisfy 
    $m_p \equiv m_1 - 1 \Mod p$.
Let $\alpha \in \ZZ$ and consider the function 
    $n \mapsto \alpha^n \crp{L}{n}$. 
This function has $m$-values
$$
m_k = \begin{cases}
\alpha^k &\text{if }k \in L, \\
0 &\text{otherwise.}
\end{cases}
$$ 
Let $\charfn{L}$ denote 
    the characteristic function of $L$, 
    so that $m_k = \alpha^k \charfn{L}(k)$. 
Then $n \mapsto \alpha^n \crp{L}{n}$ 
    is $p$-adically continuous if and only if 
    $m_1 - m_p \equiv 1 \Mod p$, i.e. 
    $$\alpha \charfn{L}(1) - \alpha^p \charfn{L}(p)
    \equiv \alpha( \charfn{L}(1) - \charfn{L}(p))
    \equiv 1 \Mod p .$$
The theorem now follows easily from 
    this congruence for $\alpha \in \{\pm 1\}$. 
%
%
%
\end{proof}


\section{Cyclic arrangements and derangements}\label{sec:three} 

In this section we study the $p$-adic properties  
of $\gs{n}{r}$ 
and demonstrate a formula for $\gs{n}{r}$ involving the floor function 
(Theorem~\ref{thm:floorfunctionformula}). 
We prove that 
$\gs{n}{r}$ is $p$-adically continuous 
as a function of two variables 
(Theorem~\ref{thm:gsjointcontinuity}), 
and give a precise description of 
the $p$-adic regularity of $n \mapsto \gs{n}{r}$ 
in terms of $|r|_p$ 
(Theorem~\ref{thm:gsinterpolates}). 
The results of this section will be used to construct 
our $p$-adic analogue of the incomplete gamma function 
in \S\ref{sec:picg}. 

\subsection{Wreath products}\label{sec:wreath} 
Let $r$ and $n$ be positive integers. 
Let $C_r$ denote the cyclic group of order $r$. 
The symmetric group $S_n$ acts by precomposition on the set 
$\mathrm{Fun}([n],C_r) = \{ f \colon [n] \to C_r\}$ of functions, 
which is itself a group under pointwise operations. 
The wreath product is defined by 
$$C_r \wr S_n := \mathrm{Fun}([n],C_r) \rtimes S_n.$$  
This group has a natural permutation action on the set $C_r \times [n]$ 
given by 
\begin{equation}\label{eqn:wreathaction}
(f,\sigma)(s,m) = (s+f(\sigma(m)),\sigma(m)).
\end{equation}

The group $C_r \wr S_n$ can also be represented as  
the group of complex $n\times n$ matrices 
generated by diagonal matrices of order $r$  
and permutation matrices. 
In this representation,
the action on the set $C_r \times [n]$ defined above
is isomorphic to matrix multiplication on 
the set 
$\{ \zeta^i e_j :  i =0, 1,\ldots, r-1,\, j=1, 2, \ldots, n\}$
where $\zeta$ is a primitive $r$th root of unity and $e_j$
is a standard basis vector of $\CC^n$.

\begin{definition}
An \defn{$r$-cyclic derangement of degree $n$} is
an element of $C_r \wr S_n$ whose action on $C_r\times[n]$ has no fixed point. 
An \defn{$r$-cyclic arrangement of degree $n$} is 
a choice of subset $A \subset [n]$ and 
an element of $C_r \wr S_{|A|}$. 
\end{definition}


\begin{example}
\label{ex:arrangements}
There are five $2$-cyclic derangements of degree $2$.
They are given by the matrices
\begin{equation} 
\begin{pmatrix}
0 & 1 \\
1 & 0 \\
\end{pmatrix}  
,
\begin{pmatrix}
0 & -1 \\
1 & 0 \\
\end{pmatrix}  
,
\begin{pmatrix}
0 & 1 \\
-1 & 0 \\
\end{pmatrix}  
,
\begin{pmatrix}
0 & -1 \\
-1 & 0 \\
\end{pmatrix}  
,
\begin{pmatrix}
{-1} & 0 \\
0 & -1 \\
\end{pmatrix}  
.
\end{equation} 
There are three $2$-cyclic arrangements of degree $1$.
In the notation $(A,f)$ with $f \in C_r \wr S_{|A|}$ 
    in its matrix form, 
    they are given by 
    $\big(\varnothing,()\big)$, 
    $\big([1],(1)\big)$, 
    $\big([1],(-1)\big)$. 
\end{example} 

Recall that we have defined the combinatorial function 
(cf.~\eqref{eqn:defngs}) 
\begin{equation}
    \gs{n}{r} = 
    \begin{cases}
        \text{$\#$ of $r$-cyclic arrangements of degree $n$} 
            & \text{if $r > 0$,} \\
        (-1)^n (\text{$\#$ of $|r|$-cyclic derangements of degree $n$}) 
            & \text{if $r < 0$.}
    \end{cases}
\end{equation} 
We begin by deriving a formula for $\gs{n}{r}$ 
    which makes sense for nonintegral values of $r$. 

\begin{lemma}
For any nonzero integer $r$ and positive integer $n$, 
    \begin{equation}\label{eqn:icumahler} 
    \gs{n}{r} = 
    \sum_{k=0}^{n} 
    \binom{n}{k} k! {r}^{k}. 
    \end{equation} 
\end{lemma}

\begin{proof} 
We use a standard inclusion-exclusion argument. 
When $r$ is positive the formula follows  
    by summing $\#(C_r \wr S_k) = k!r^k$ 
    over the possible cardinalities $k$ of 
    the $\binom{n}{k}$ choices of subset $A \subset [n]$.
Suppose $r$ is negative. 
Observe from \eqref{eqn:wreathaction} that 
    the set of fixed points of $(f,\sigma)$ is equal to 
$C_{|r|} \times A$ where $A = \{m \in [n] : \sigma(m) = m,\, f(m)=0\}$. 
Then 
\begin{align*}
\#(\text{$|r|$-cyclic derangements})
&= \#( x \in C_{|r|} \wr S_n) - \sum_{i\in [n]} \#( x \text{ which fix }C_{|r|} \times \{i\}) \\
& \qquad + \sum_{ \{i,j\} \subset [n]}\#( x \text{ which fix } C_{|r|} \times \{i,j\}) - \cdots \\
&= \sum_{k= 0}^n (-1)^k \sum_{\substack{A \subset [n] \\ |A|=k}} \#( x \in C_{|r|} \wr S_n \text{ which fix }C_{|r|}\times A) .
\end{align*}
For any subset $A \subset [n]$, 
    the subgroup of $C_{|r|} \wr S_n$ 
    fixing every element of $C_{|r|} \times A$ 
    is $C_{|r|} \wr S_{n-\# A}$ 
    which has cardinality $(-r)^{n-\# A} (n-\# A)!$. 
Thus $(-1)^n \gs{n}{r}$ is equal to 
\begin{align*} 
    \#(\text{$|r|$-cyclic derangements}) &= 
    \sum_{k= 0}^n (-1)^k \binom{n}{k} (-r)^{n-k} (n-k)! \\
    &= 
(-1)^n \sum_{k=0}^n \binom{n}{k}r^k k!.\qedhere
\end{align*} 
\end{proof} 

We henceforth use \eqref{eqn:icumahler} to define $\gs{n}{r}$ 
for all $r \in \CC_p$ and $n \in \NN$. 



\subsection{\texorpdfstring{$p$}{p}-adic interpolation of \texorpdfstring{$\gs{n}{r}$}{a(n,r)}} 

In this section we prove that $\gs{n}{r}$ is interpolated 
by a two-variable $p$-adically continuous function. 
First we interpolate 
$n \mapsto \gs{n}{r}$ for fixed $r \in \CC_p^\times$ 
(Theorem~\ref{thm:gsinterpolates}), 
and afterwards we interpolate $\gs{n}{r}$ 
in both variables 
(Theorem~\ref{thm:gsjointcontinuity}). 





\begin{theorem}\label{thm:gsinterpolates} 
Let $r\in\CC_p^\times$ 
    and let $\gs{\,\cdot\,}{r} \colon \NN \to \CC_p$ be defined by \eqref{eqn:icumahler}.
Then $n \mapsto \gs{n}{r}$ extends to 
    an analytic function 
    (resp. metric map, 
    locally analytic function) 
    on the closed unit disk 
    if and only if $|r|_p < 1$ 
    (resp. $|r|_p \leq 1$, 
    $|r|_p < p^{(p-1)^{-1}}$).  
    If $|r|_p \geq p^{(p-1)^{-1}}$ then 
    $n \mapsto \gs{n}{r}$ 
    does not extend to 
    a continuous function on $\ZZ_p$. 
\end{theorem} 


\begin{corollary} 
For any prime $p$ and integer $r$, 
    the function $n \mapsto \gs{n}{r}$
    extends to a 
    locally analytic metric map on $\ZZ_p$. 
    If $p$ divides $r$, 
    then $n \mapsto \gs{n}{r}$ extends to an 
    analytic metric map on $\ZZ_p$. 
\end{corollary} 

\begin{remark}\label{rmk:anrmkpinterpolation}
We note that $n \mapsto \gs{n}{r}$ 
has integral $m$-values if $r \in \ZZ_p$ (cf.~\eqref{eqn:mvalues}), 
    and is therefore encompassed 
    by Theorem~\ref{thm:mkpinterpolation}. 
Indeed from \eqref{eqn:productbinomialconvolution} we have 
\begin{multline}
    \sum_{n=0}^\infty \gs{n}{r} \frac{X^n}{n!}
    =
    e^X \cdot 
    \sum_{n=0}^\infty 
        n! {r}^{n} 
        \frac{X^n}{n!}
        = e^X (1-rX)^{-1}\\
        = \exp\left( (1+r)\tfrac{X}{1} + r^2\tfrac{X^2}{2}+ 
        r^3\tfrac{X^3}{3} + \cdots \right)
\end{multline}
from which we can read off the $m$-values. 
The required congruence \eqref{eqn:nvaluescongruence} 
for continuity is satisfied: 
    $m_p = r^p \equiv r = (r+1) -1 \Mod{p\ZZ_p}$. 
\end{remark}


\begin{proof}[Proof of Theorem~\ref{thm:gsinterpolates}] 


The $k$th finite difference of 
    $n \mapsto \gs{n}{r}$ is $\findiff[k] = r^{k}k!$. 
    By Theorem~\ref{thm:allinterpolations}, $\gs{n}{r}$ extends to a 
    $p$-adic analytic function on the closed unit disk if and only if 
    $|c(n)/n!|_p = |r|_p^n \to 0$ as $n \to \infty$, and this occurs 
    if and only if $|r|_p < 1$. This proves the first equivalence. 
   %

    Now observe that $p^{\floor{\frac{\log n}{\log p}} }$ is 
    the largest power of $p$ that is less than or equal to $n$. 
    It follows that 
    $$
    \sup_{n \geq 1} 
        |n!|_p\, p^{\floor{\frac{\log n}{\log p}} } \leq 1. 
    $$
    Clearly, if $|r|_p \leq 1$ then we have 
    $$
    \sup_{n \geq 1} 
        |\findiff[n]|_p\, p^{\floor{\frac{\log n}{\log p}} } 
        =
    \sup_{n \geq 1} 
        |r^nn!|_p\, p^{\floor{\frac{\log n}{\log p}} } 
        \leq 1, 
    $$
    in which case $\gs{n}{r}$ extends to a metric map on $\ZZ_p$ 
    by Theorem~\ref{thm:allinterpolations}. 
    On the other hand, if $|r|_p > 1$ then 
    $$
    \sup_{n \geq 1} 
        |r^nn!|_p\, p^{\floor{\frac{\log n}{\log p}} }
       \geq |r|_p > 1, 
    $$
    showing that 
    $\gs{n}{r}$ does not extend to a metric map on $\ZZ_p$. 

    Now we prove the third equivalence. 
    Using Legendre's formula for $v_p(n!)$ 
    we obtain 
    $\valn[p](\findiff[n]) = 
        n\left((p-1)^{-1} + \valn[p](r)\right) - s_p(n)(p-1)^{-1}$ 
    where $s_p(n)$ is the sum of the digits of $n$ in base-$p$. 
    In particular, we have 
    \begin{equation}\label{eqn:mahlerseq2} 
        \valn[p](\findiff[n]) = 
        n\left((p-1)^{-1} + \valn[p](r)\right) - O(\log n).
    \end{equation} 
    Thus 
    $$
    \liminf_{n \to \infty} 
    \left(\frac{\valn[p](\findiff[n])}{n}\right)
    =(p-1)^{-1}+\valn[p](r).
    $$
    By Amice's theorem (Theorem~\ref{thm:allinterpolations}), 
    $\gs{n}{r}$ extends to 
    a $p$-adic locally analytic function 
    on the closed unit disk  
    if and only if 
    $(p-1)^{-1}+\valn[p](r)>0$, 
    or equivalently, $|r|_p < p^{(p-1)^{-1}}$.

    For the last claim, 
    from \eqref{eqn:mahlerseq2} it is clear that 
    $\valn[p](\findiff[n]) \to -\infty$ 
    if $(p-1)^{-1} + \valn[p](r) < 0$, or 
    equivalently, if $|r|_p > p^{(p-1)^{-1}}$. 
    On the other hand, if $(p-1)^{-1} + \valn[p](r) = 0$ 
    then 
    \begin{equation}
        \valn[p](\findiff[n]) = - s_p(n)(p-1)^{-1}. 
    \end{equation}
    Since $s_p(p^N) = 1$ for any positive integer $N$, 
    we see that in this case $\valn[p](\findiff[n])$ 
    does not go to $\infty$ as $n \to \infty$. 
    We conclude that $\gs{n}{r}$ does not 
    admit an extension to a continuous function 
    on $\ZZ_p$ when $|r|_p \geq p^{(p-1)^{-1}}$. 
\end{proof} 

Recall a two-variable function 
$f \colon \ZZ_p \times \ZZ_p \to \ZZ_p$ is a metric map 
if for all $x,y,r,s$, 
\begin{equation} 
    |f(x,r)-f(y,s)|_p \leq \max\{ |x-y|_p,|r-s|_p\}.
\end{equation} 

\begin{theorem}\label{thm:gsjointcontinuity} 
The function 
    $\NN \times \NNpos \to \ZZ_p 
    \colon (n,r) \mapsto \gs{n}{r}$ extends to 
    a metric map $\tilde{a} \colon \ZZ_p \times \ZZ_p \to \ZZ_p$. 
    Explicitly, it is given by the convergent series  
\begin{equation}\label{eqn:gsjointcontinuityformula} 
    \gsp{y}{r} 
= 1+ \sum_{k=1}^\infty r^k y(y-1)\cdots(y-k+1). 
\end{equation} 
\end{theorem} 

\begin{proof} 
Let $(x,r),(y,s)\in\ZZ_p \times \ZZ_p$. 
    The function $n \mapsto \gs{n}{r}$ admits a 
    continuous extension in the first variable (denoted $\gsp{\,\cdot\,}{r}$) 
    by Theorem~\ref{thm:gsinterpolates}.  
From \eqref{eqn:icumahler}, 
    the Mahler series of $\gsp{\,\cdot\,}{r}$ is 
\begin{equation}\label{eqn:alpha} 
    \gsp{y}{r} = \sum_{k=0}^\infty r^kk! \binom{y}{k}
    = 1+ \sum_{k=1}^\infty r^k y^{\ul k} 
\end{equation} 
where $y^{\ul k} := y(y-1) \cdots (y-k+1)$.
This sum is absolutely convergent since $\gsp{\,\cdot\,}{r}$ is continuous. 
This shows that the assignment $(y,r) \mapsto \gsp{y}{r}$ is well-defined. 

    By absolute convergence we can rearrange \eqref{eqn:alpha} to obtain 
\begin{align}\label{eqn:differenceinsecondvariable}
    |\gsp{y}{r} - \gsp{y}{s}|_p 
        &= \Big|\sum_{k=1}^\infty  (r^k  -s^k )y^{\ul k} \Big|_p\\
        &\leq \max_{k\geq 1} \big(|r^k-s^k|_p \,  |y^{\ul k}|_p \big)\\
        &\leq |r-s|_p. 
\end{align}
In the last inequality, we have used that 
    $|r^k - s^k|_p \leq |r-s|_p$
and $|y^{\ul k}|_p \leq 1$
for $k\geq 1 $ and $y,r,s\in \ZZ_p$.

We have the inequality 
    $|\gsp{x}{r} - \gsp{y}{r}|_p \leq |x - y|_p$ 
    since $\gsp{\,\cdot\,}{r}$ is a metric map by Theorem~\ref{thm:gsinterpolates}. 
Combining this inequality with \eqref{eqn:differenceinsecondvariable} 
we obtain 
\begin{align} 
    |\gsp{x}{r} - \gsp{y}{s}|_p 
    &\leq \max\{ |\gsp{x}{r}-\gs{y}{r}|_p,|\gsp{y}{r}-\gsp{y}{s}|_p \} \\
    &\leq \max\{ |x-y|_p,|r-s|_p \}. 
\end{align} 
This inequality completes the proof. 
\end{proof} 


\subsection{An archimedean formula for \texorpdfstring{$\gs{n}{r}$}{a(n,r)}}\label{sec:floorformulaanr} 


The conclusions of 
Theorem~\ref{thm:gsinterpolates} and 
Theorem~\ref{thm:gsjointcontinuity} 
are rather surprising 
in view of the following beautiful formula for $\gs{n}{r}$ 
in terms of the floor function. 

\begin{theorem}\label{thm:floorfunctionformula}
Let $r$ be a nonzero integer and let $n$ be a positive integer. 
Then 
\begin{equation}
    \label{eqn:explicitalgu}
    \gs{n}{r} = 
\begin{cases}
    \floor{e^{1/{r}}  r^n n! + \tfrac{1}{2}}
        & \text{if $r<0$,}\\
    \floor{e^{1/{r}}r^n n!} 
        & \text{if $r>0$.}
\end{cases}
\end{equation} 
\end{theorem}

The cases $r \in \{\pm 1\}$ are well-known 
(e.g.~\cite{hassani-2003}). 
In view of its $p$-adic properties, 
the function $\gs{n}{r}$ thus exhibits 
a surprising harmony between the $p$-adic and 
real topologies, the latter of which 
is used in \eqref{eqn:explicitalgu} 
to define the ordering on $\RR$. 

\begin{proposition} 
    \label{prop:hallsequencespositive}
Let $r$ be a positive integer and let
$\gs{n}{r}$ denote the number of $r$-cyclic arrangements of degree $n$.
Then 
\begin{equation}
    \gs{n}{r} = 
    \floor{e^{\tfrac 1{r}}r^n n!} \quad (n \in \NNpos). 
\end{equation}
\end{proposition} 

\begin{proof} 
%
Observe that 
\begin{align}\label{eqn:ean} 
    e^{\tfrac 1{r}}  r^n n!
    &= 
    r^n n!\, 
    \Big(1 + r^{-1} + \frac{r^{-2}}{2!} + 
    \frac{r^{-3}}{3!} + \cdots \Big)\\
    &=r^n n! + \nonumber
    r^{n-1} (n-1)! \binom{n}{1} + 
    r^{n-2}(n-2)!\binom{n}{2} +
    \cdots +
    r \binom{n}{n-1} + 1 \\
    &\quad\quad  \nonumber
    + \frac{r^{-1}}{n+1}
    + \frac{r^{-2}}{(n+1)(n+2)}
    + \cdots .
\end{align} 
Let $R_{r,n}$ denote the `remainder term'
\begin{equation}
\label{eqn:remainder}
R_{r,n} = \sum_{k\geq 1} \frac{n!}{(n+k)!} r^{-k} 
= \frac{r^{-1}}{n+1}
    + \frac{r^{-2}}{(n+1)(n+2)} + \cdots. 
\end{equation}
We have 
$e^{1/r}r^n n! = \gs{n}{r} + R_{r,n}$.
Each summand of $R_{r,n}$
is a decreasing function of $r$ and a decreasing function of $n$.
Since $r,n\geq 1$, we have
$R_{r,n}\leq R_{1,1} =  e - 2< 1.$
It follows that  
\begin{align*}
    \floor{e^{\tfrac 1{r}}  r^n n!} 
    &= 
        \floor{\gs{n}{r} + R_{r,n}}
        = \gs{n}{r}.\qedhere
\end{align*}
\end{proof} 

\begin{proposition} 
    \label{prop:hallsequencesnegative}
Let $r$ be a negative integer and let
$
\gs{n}{r}$
denote the quantity
$ (-1)^n (\text{number of $|r|$-cyclic derangements of degree $n$}).
$ 
Then 
\begin{equation}
    \gs{n}{r} = 
    \floor{e^{\tfrac 1{r}}r^n n!+\tfrac12} \quad (n \in \NNpos). 
\end{equation}
\end{proposition} 

\begin{proof} 
Continue letting $R_{r,n}$ be as in \eqref{eqn:remainder} 
so that $e^{1/r} r^n n! = \gs{n}{r} + R_{r,n}$.
For all positive integers $n$ we have 
\begin{align}
\label{eq:summand}
- R_{r,n} &= \frac{|r|^{-1}}{n+1} - \frac{|r|^{-2}}{(n+1)(n+2)} + \frac{|r|^{-3}}{(n+1)(n+2)(n+3)} - \cdots \nonumber \\
&= 
 \sum_{\ell = 1}^\infty
\frac{|r|^{-2\ell+1}}{(n+1)(n+2) \cdots (n+2\ell)}\Big(
    (n + 2\ell) - {|r|^{-1}}
\Big) . 
\end{align}
Clearly each summand of \eqref{eq:summand} is positive, so $R_{r,n}$ is negative. 
We will show that $|R_{r,n}|$ is a decreasing function of $n$
by considering each summand individually. 

Regarding the summand of \eqref{eq:summand} as a function of 
a continuous variable $n$, 
the expression has log-derivative 
\begin{multline}
\label{eqn:logderivative} 
\frac{d}{dn}\Big(
    \log(|r|^{-2\ell+1})
    +
    \log(n +2\ell  - |r|^{-1}) 
    - \sum_{k=1}^{2\ell} \log(n+k) 
\Big) \\ 
= \frac1{n+2\ell - |r|^{-1}} - \sum_{k=1}^{2\ell} \frac1{n+k} .
\end{multline}
Since $|r| \geq 1$, we have the bounds
\begin{equation} 
\frac1{n+2\ell - |r|^{-1}} \leq \frac1{n + 2\ell - 1} < 
\sum_{k=1}^{2\ell} \frac1{n+k} , 
\end{equation} 
which imply that \eqref{eqn:logderivative} is negative. 
Therefore each (positive) summand in 
\eqref{eq:summand} is   decreasing as a function of $n$.

This shows 
in particular that
$|R_{r,n}| \leq |R_{r,1}|$. 
We have 
$$
|R_{r,1}| = 1+r - r e^{\tfrac{1}{r}}
\in
(0,\tfrac12)
\quad\text{for } r \leq -1. 
$$
Thus we have 
$
e^{\tfrac 1{r}}  r^n n! = 
\gs{n}{r} + R_{r,n}
$ 
with  
$R_{r,n} \in 
(-\tfrac12 , 0). $
This shows that for all $n \geq 1$,
\begin{equation} 
    \floor{e^{\tfrac 1{r}}  r^n n! + \tfrac{1}{2}}
    = \floor{\gs{n}{r} + R_{r,n} + \tfrac12}
    = \gs{n}{r}.
\tag*{\qedhere} 
\end{equation} 
\end{proof} 



\section{The \texorpdfstring{$p$}{p}-adic incomplete gamma function}\label{sec:picg} 

In 1975, Morita \cite{morita-1975} discovered 
a $p$-adic counterpart to the gamma function. 
In this section, we show that 
the incomplete gamma function also 
has a $p$-adic counterpart.

\subsection{The classical incomplete gamma function} 

We recall the classical construction of 
the incomplete gamma function. 
For a complex number $s$ 
let $\omega_s$ denote 
the differential 
given by 
\begin{equation} 
    \omega_s(t) := t^s e^{-t} \frac{dt}{t} 
\end{equation} 
where we take the principal value for 
$t^s= \exp(s \log t)$.
We choose the branch cut $(-\infty,0]$
so that $\omega_s$ is holomorphic and single-valued on the domain
$
U = \CC \setminus (-\infty,0] . 
$ 
For any oriented half-infinite curve $\gamma$ in $U$ 
tending towards positive infinity, 
the integral 
$ \int_\gamma \omega_s $
converges and defines a holomorphic function of $s$. 
We thus obtain, for any complex number $z\in U$, 
the entire function of $s$ given by  
\begin{equation}\label{eqn:incg-classical} 
    \Gamma(s,z):= \int_\gamma \omega_s 
\end{equation} 
where $\gamma$ is any curve in $U$ beginning at $z$ 
and tending to positive infinity. 
This defines the \defn{incomplete gamma function}
$\Gamma\colon \CC \times U \to \CC$. 
By Hartog's theorem on separate analyticity, 
$\Gamma(s,z)$ is 
an analytic function of two variables. 

The incomplete gamma function recovers 
the gamma function in the limit as $z$ goes to zero: 
$$\lim_{z \to 0^+} \Gamma(s,z) = \Gamma(s). 
$$ 

In general, the incomplete gamma function 
takes different values as $z$ approaches 
the branch cut. 
For $s \in \CC$ and $z$ on 
the negative real axis we have 
\begin{equation}\label{eqn:branchcutjump}
\Gamma(s, z + i 0^+) = 
e^{2\pi i s} \Gamma(s, z - i0^+) + (1 - e^{2\pi i s})\Gamma(s) 
\end{equation}
where 
$\displaystyle f(0^+) \coloneqq 
\lim_{\epsilon \to 0^+} f(\epsilon)$; 
see 
\cite[Eq. 8.2.9]{nist}.
If $s$ is a positive integer 
this shows $\Gamma(s,z)$ 
extends continuously in the second variable 
across the branch cut,  
and $\Gamma$ extends to the domain 
$\Gamma\colon (\CC \times U) \cup (\NNpos \times \CC) \to \CC$. 

The incomplete gamma function 
is closely related to the combinatorial function $\gs{n}{r}$ 
counting cyclic derangements and arrangements.
Just as the count of all permutations 
is interpolated by the gamma function, 
the counts of cyclic derangements and cyclic arrangements 
are interpolated by the incomplete gamma function 
(up to an explicit entire function).  
Recall from \S\ref{sec:three} that
we extended the combinatorial definition of $\gs{n}{r}$ to 
$\gs{n}{r} = 
\sum_{k=0}^{n} 
    \binom{n}{k} k! {r}^{k}$
for all $r \in \CC_p$. 

\begin{proposition}\label{prop:icggs} 
    Let $z \in \CC^\times$ be fixed. 
    The function 
    $s \mapsto \icg[1+s, z] 
        z^{-s}
        e^{z}$ 
    is holomorphic for all $s \in \CC$. 
For $n \in \NN$ we have 
    $\icg[1+n, z] 
        z^{-n}
        e^{z} = \gs{n}{1/z}.$ 
\end{proposition} 

($\Gamma(1+n,z)$ is defined 
for $(n,z) \in \NN \times \CC$ by 
\eqref{eqn:branchcutjump}.) 

\begin{proof} 
The function $z^{-s} = \exp({-s \log z})$ is entire
in the argument $s$,  
where we take the principal value for the logarithm 
    to determine the constant $\log z$,  
and $\icg[1+s, z]$ is entire in $s$. 
This proves the first part of the proposition.    
    
By \cite[Eq.~8.7.3]{nist}, 
    for any $(s,z) \in \NNpos \times \CC$ 
    we have 
\begin{equation} 
    \icg[s,z] = 
    \Gamma(s)
    \Bigg(
    1 - z^se^{-z}
        \sum_{k=0}^\infty
            \frac{z^k}{\Gamma(s+k+1)}
    \Bigg).
\end{equation} 
    Using that $\Gamma(s+k+1)\Gamma(s)^{-1} = (s+k)\cdots(s+1)s$ 
    and writing $s = 1+n$ we have 
\begin{align} 
    \icg[1+n,z] 
    &= 
    n! 
    \Bigg(
    1 - z^{1+n}e^{-z}
        \sum_{k=0}^\infty
            \frac{z^{k}}{(n+k+1)!}
    \Bigg)\\
    &= 
    n! 
    \Bigg(
    1 - e^{-z}
        \sum_{k=n+1}^\infty
            \frac{z^{k}}{k!}
    \Bigg)\\
    &= 
    n! 
    e^{-z}
    \Bigg(
    e^{z} - \sum_{k=n+1}^\infty 
                \frac{z^{k}}{k!}
    \Bigg)\\
    &=
    n! 
    e^{-z}
    \sum_{k=0}^n  
        \frac{z^{k}}{k!} \\
    &={e}^{-z} 
    \sum_{k=0}^{n} 
        \frac{ n!}{(n-k)!}z^{n-k}
    =
    \gs{n}{1/z}
    z^{n}
    e^{-z}. 
\end{align} 
Rearranging terms proves the proposition. 
\end{proof} 


\begin{corollary}\label{cor:Qezfinitedata}
For any $z \in \CC^\times$, $\QQ(\{\icg[n, z] : n \in \NN^{>0}\}) = \QQ(z,e^z)$.
\end{corollary}

\begin{proof}
From Proposition~\ref{prop:icggs} it is easily seen that for any positive integer $n$, 
$\icg[n,z]$ is a polynomial function in $z$ and $e^z$. 
This shows that 
$\QQ(\{\icg[n, z] : n \in \NN^{>0}\}) \subset \QQ(z,e^z)$. 
Moreover, 
$e^{-z}=\icg[1,z]$ and 
    $z = \tfrac{\icg[2,z]}{\icg[1,z]} - 1$,
proving that 
$\QQ(z,e^z) \subset \QQ(\{\icg[n, z] : n \in \NN^{>0}\})$. 
\end{proof}


\subsection{The \texorpdfstring{$p$}{p}-adic incomplete gamma function} 

We will interpolate the incomplete gamma function 
on its integral values $(n,r) \in \NNpos\times \ZZ$. 
Before doing so, it is necessary to realize 
the integral values of the incomplete gamma function 
(which are transcendental elements of $\CC$) 
as $p$-adic numbers. 
As we shall see, there is 
a reasonably natural way of doing this. 
The field generated by the integral values of the incomplete gamma function is 
$$\QQ(\{\Gamma(n,r) : (n,r) \in \NNpos\times\ZZ\}) = \QQ(e)$$ 
by Corollary~\ref{cor:Qezfinitedata}. 
Therefore, to obtain a homomorphism from the field generated by 
$\{\Gamma(n,r) : (n,r) \in \NNpos\times\ZZ \}$ 
to the field of $p$-adic numbers, 
it suffices to specify a transcendental element of $\QQ_p$ corresponding to $e$. 
Let 
$\tau_p \colon \QQ(e) 
\to \QQ_p$ 
be the unique field homomorphism satisfying  
\begin{equation} 
    \tau_p(e^{-1}) = \exp_p(p) = 1+ p + p^2/2!+p^3/3!+\cdots. 
\end{equation} 
The element $\exp_p(p) \in \QQ_p$ is transcendental over $\QQ$ 
(cf.~\cite[\S3]{adams1966transcendental}). 

\begin{theorem}[Theorem~\ref{thm:introincompletegamma}]
There exists a continuous function 
    $$\Gamma_{p}\colon\ZZ_p \times (1+p\ZZ_p)\to \ZZ_p$$ 
    satisfying 
$$\Gamma_{p}(n,r) = 
    \tau_p\Gamma(n,r)
    $$ 
for any positive integers $n$ and $r$ satisfying $r \equiv 1 \Mod p$. 
\end{theorem}

\begin{proof} 
We have that 
    \begin{equation}\label{eqn:factorizationunifcont} 
    \tau_p\Gamma(n,r)
    =
    \tau_p\big(\gs{n-1}{1/r}
    r^{n-1}
    e^{-r} \big)
    =
    \gs{n-1}{1/r}
    r^{n-1}
    \tau_p(e^{-r} ) 
\end{equation} 
by Proposition~\ref{prop:icggs}. 
We will show that each of the three functions 
    $(n,r) \mapsto \gs{n-1}{1/r}$, $(n,r) \mapsto r^{n-1}$, and 
    $r \mapsto \tau_p(e^{-r})$ is interpolated by a uniformly continuous function
    in the $p$-adic topology. 

We have already proven  
    in Theorem~\ref{thm:gsjointcontinuity}
    that $(n,r) \mapsto \gs{n}{r}$ 
    extends to the metric map 
    $\ZZ_p \times \ZZ_p \to \ZZ_p$ 
    given by the Mahler series \eqref{eqn:alpha} 
    \begin{equation} 
    \gs{y}{s} = \sum_{k=0}^\infty s^kk! \binom{y}{k} \,\,\,
    \text{for $(y,s)\in\ZZ_p \times \ZZ_p$}. 
    \end{equation} 
This suffices to prove that 
    $(n,r) \mapsto \gs{n-1}{1/r}$ 
    is interpolated by a metric map, 
    hence interpolated by a uniformly continuous map. 

Let $\exp_p \colon \ZZ_p \to 1+p\ZZ_p$ (resp. $\log_p \colon 1+p\ZZ_p \to \ZZ_p$) denote 
the $p$-adic exponential (resp. $p$-adic logarithm). 
Then, for any positive integers $n$ and $r$ satisfying $r \equiv 1 \Mod p$,  
    we have that $r^n = \exp_p(n \log_p r)$. 
Therefore the function 
$\ZZ_p \times (1+p\ZZ_p) \to \ZZ_p\colon (x,s) \mapsto \exp_p(x \log_p s)$ 
is a continuous extension of 
    $\NNpos \times (1+p\NN) \to \ZZ_p\colon (n,r) \mapsto r^n$ 
    which is also uniformly continuous since $\ZZ_p \times (1+p\ZZ_p)$ is compact. 

Finally,
    $\NNpos \to \ZZ_p \colon r \mapsto \tau_p(e^{-r}) = \exp_p(p)^r$ 
    is interpolated by the $p$-adic locally analytic function 
    $\ZZ_p \to \ZZ_p\colon r \mapsto \exp_p(pr)$ 
    which is also uniformly continuous since $\ZZ_p$ is compact. 

Any uniformly continuous function $f \colon U \to Y$  
mapping to a complete metric space $Y$, 
    and defined on a dense subset $U \subset X$ 
    of a topological space $X$, 
    extends uniquely to a uniformly continuous function $X \to Y$. 
We conclude that $\tau_p\Gamma(n,r)$ is interpolated by 
    the uniformly continuous function 
\begin{equation}\label{eqn:factorizationforpicg} 
    \Gamma_p(y,r) = r^{y-1} \exp_p(pr) a(y-1,1/r) 
\end{equation} 
for $(y,r) \in \ZZ_p \times(1+p\ZZ_p)$.
\end{proof} 


\begin{remark}[$m$-values of $\Gamma_p$]
We can corroborate Theorem~\ref{thm:introincompletegamma} 
    using Theorem~\ref{thm:mkpinterpolation}. 
For fixed $r \in \ZZ$ 
the $m$-values of $n \mapsto \exp_p(-pr) \tau_p \Gamma(n,r)$ 
    are found to be 
\begin{equation} 
m_k = 
\begin{cases}
    r+1 & \text{if $k = 1$,}\\
    1 & \text{if $k \geq 2$,} 
\end{cases}
\end{equation} 
and the required congruence 
    \eqref{eqn:nvaluescongruence} 
    for continuity is that 
    $r \equiv 1 \Mod p$. 
\end{remark}

We collect a few facts about $\Gamma_p$. 

\begin{theorem}
Let $\Gamma_{p}\colon\ZZ_p \times (1+p\ZZ_p)\to \ZZ_p$ denote 
    the $p$-adic incomplete gamma function 
    of Theorem~\ref{thm:introincompletegamma}. 
    \begin{enumerate}
        \item For any $y \in \ZZ_p$ and $r \in 1 + p\ZZ_p$ we have 
            the series expression 
    \begin{equation} 
        \Gamma_p(y,r) = \sum_{k=0}^\infty r^{y-k-1}\exp_p(pr)  k! \binom{y-1}{k}. 
    \end{equation} 
        \item For any $n\in \NN$ and $r \in \ZZ$ satisfying $r \equiv 1 \Mod p$, we have   
    \begin{equation} 
        \Gamma_p(n+1,1/r) = 
        r^{-n}\exp_p(p/r)  \cdot
        \begin{cases}
            \floor{e^{1/r}r^nn!+\tfrac12} &\text{if $r<0$,} \\
            \floor{e^{1/r}r^nn!} &\text{if $r>0$.}
        \end{cases}
    \end{equation} 
        \item For any $n\in \NN$ and $r \in 1+p\ZZ_p$, we have   
    \begin{equation} 
        \Gamma_p(n+1,r) = 
        \exp_p(pr) n! f_n(r)
    \end{equation} 
            where $f_n(X)$ is the $n$th \emph{truncated exponential polynomial}
    \begin{equation} 
        f_n(X) := 1 + X + \tfrac{1}{2!}X^2 + \cdots + 
        \tfrac{1}{n!}X^n
        \quad(n \geq 1). 
    \end{equation} 
    \end{enumerate}
\end{theorem}

\begin{proof}
The first formula follows at once from 
    \eqref{eqn:factorizationforpicg} and \eqref{eqn:alpha}. 
The second formula follows from \eqref{eqn:factorizationforpicg} and 
    \eqref{eqn:explicitalgu}. 
The third formula follows from \eqref{eqn:factorizationforpicg} and 
    the formula 
    \begin{equation}
    \gs{n}{r} = 
    \sum_{k=0}^{n} 
        \frac{n!}{(n-k)!} {r}^{k} 
        = 
        r^{n} 
        n!
    \sum_{k=0}^{n} 
        \frac{1}{k!} {r}^{-k} 
        =
        r^{n} 
        n!
        f_n(1/r).\tag*{\qedhere} 
    \end{equation} 
\end{proof}

Using a classical result of Schur 
on the irreducibility of 
truncated exponential polynomials, 
we can make the following observation.  

\begin{theorem}
$\Gamma_p$ has finitely many zeros in $\ZZ_p \times (1+p\ZZ_p)$. 
\end{theorem}

Observe that if $(y,r) \in \ZZ_p \times(1+p\ZZ_p)$ is a zero 
of $\Gamma_p$, then the formula 
$\Gamma_p(y,r) = r^{y-1} \exp_p(pr) \gs{y-1}{1/r}$ 
shows that $(y-1,1/r)$ is a zero of $a$ 
(the formulas $1 = r^yr^{-y} = \exp_p(x)\exp_p(-x)$ 
show that the other terms never vanish). 
The result now follows from the next proposition. 


\begin{proposition}\label{prop:finitezeros} 
    Let $r \in \CC_p^\times$ and suppose $|r|_p < p^{(p-1)^{-1}}$. 
    Let $\gsp{\,\cdot\,}{r}$ denote the unique continuous function 
    $\ZZ_p \to \CC_p$ which extends \eqref{eqn:icumahler} 
    (this function exists by Theorem~\ref{thm:gsinterpolates}). 
    Then the set $S=\{s \in \ZZ_p : \gsp{s}{r} = 0\}$ is finite. 
\end{proposition} 

A nonzero analytic function on a closed region has only finitely many zeros, 
in which case the result follows immediately from Theorem~\ref{thm:gsinterpolates} 
when $|r|_p<1$. 
However, a nonzero $p$-adic \emph{locally} analytic function may have 
infinitely many zeros on a closed region. 
Thus, for $1 \leq |r|_p < p^{(p-1)^{-1}}$, the function 
$s \mapsto \gsp{s}{r}$ 
may have infinitely many zeros in $\ZZ_p$ a priori. 
The proposition shows that this does not occur. 


\begin{proof} 
    By Theorem~\ref{thm:gsinterpolates}, $\gsp{\,\cdot\,}{r}$ is locally analytic on 
    the closed unit disk. 
This means that there is a covering of $\ZZ_p$ by 
    open disks 
    such that the restriction of $\gsp{\,\cdot\,}{r}$ to 
    any one of them is 
    given by a convergent power series with coefficients in $\CC_p$. 
    As $\ZZ_p$ is compact, 
    we may take this open covering to be finite. 

    Suppose that $S$ is infinite for a contradiction. 
    Then at least one disk $D$ of the open covering 
    must contain infinitely many zeros of $\gsp{\,\cdot\,}{r}$. 
    By Weierstrass's Preparation Theorem, 
    a nonzero analytic function on an open disk 
    has only finitely many zeros 
    (cf.~e.g.~\cite[\S9]{murty}). 
    Thus the restriction of $\gsp{\,\cdot\,}{r}$ 
    to $D$ is identically zero. 

    Let $m$ and $n$ be distinct integers 
    chosen from $D \cap \NN$
    (this is possible because $\NN$ is dense in $\ZZ_p$). 
    By assumption $\gsp{m}{r} = \gsp{n}{r} = 0$.
    We have 
    \begin{equation}\label{eqn:schurwithgs} 
    \gsp{n}{r} = 
    \sum_{k=0}^{n} 
        \frac{n!}{(n-k)!} {r}^{k} 
        = 
        r^{n} 
        n!
    \sum_{k=0}^{n} 
        \frac{1}{k!} {r}^{-k} 
        =
        r^{n} 
        n!
        f_n(1/r).
    \end{equation} 
    This implies that $1/r$ is a root of 
    $f_m(X)$ as well as $f_n(X)$. 
    These polynomials are both irreducible over $\QQ$ 
    by a classical theorem of Schur \cite{schur},  
    implying that the degree of $1/r$ as an algebraic number 
    is simultaneously equal to $m$ and $n$, a contradiction. 
    This implies $S$ is finite. 
\end{proof} 




\subsection{Some questions} 
The classical gamma function $\Gamma(s)$ is recovered from the 
incomplete gamma function $\Gamma(s,z)$ in 
the archimedean limit as $z \to 0^+$.  
Can Morita's $p$-adic gamma function \cite{morita-1975} or
Diamond's $p$-adic log gamma function \cite{diamond} 
be related to the $p$-adic incomplete gamma function? 
In 1984, Schikhof \cite[p.~17]{schikhof} asked (in our notation) 
whether 
$\gsp{-1}{-1} = 
(-1)^{p-1}d_{-1,p} =1 + 1 + 2! + 3! + 4! + 5! + \cdots$ 
is in $\QQ \cap \ZZ_p$. 
R.~Murty--Sumner \cite{murty-sumner} conjecture 
a negative outcome to Schikhof's question. 
More generally, we ask the same question of $\gsp{-1}{r}$. 
Is there a nonzero integer $r$ such that 
    $\gsp{-1}{r}\in \QQ \cap \ZZ_p$? 

There are interesting examples of functions $f \colon \NN \to \CC_p$ 
    whose $m$-values are not in $\ZZ_p$. 
Does Theorem~\ref{thm:intromkpinterpolation} 
    admit a generalization for arbitrary 
    functions $f \colon \NN \to \CC_p$? 
Can stronger $p$-adic regularity 
    (Lipschitz continuity, differentiability, 
    integrality of divided differences, analyticity) 
    be detected with congruences on $m$-values?

\section*{Acknowledgements} 

We thank the reviewer for their 
    useful suggestions. 
We are also grateful to 
Daniel Barsky and Julian Rosen for pointing out 
an error on an earlier version of this manuscript. 
We thank Jeffrey Lagarias for some helpful comments. 
We also thank A. Suki Dasher for a valuable suggestion which led 
to the connection with the incomplete gamma function. 


\bibliography{ig} 
\bibliographystyle{abbrv}

\end{document}